\documentclass[12pt,reqno]{amsproc}
\usepackage{amsfonts}
\usepackage{amsmath}
\usepackage{amssymb}
\usepackage{graphicx}

 \theoremstyle{plain}
\newtheorem{theorem}{Theorem}[section]

\newtheorem{definition}[theorem]{Definition}
\newtheorem{example}[theorem]{Example}

\newtheorem{lemma}[theorem]{Lemma}

\numberwithin{equation}  {section}
\begin{document}
	
	\title[A New Generalized Beurling Theorem]{A Generalized
		Beurling Theorem in Finite von Neumann Algebras}

	\author{Don Hadwin}
	\address{University of New Hampshire}
	\email{don@unh.edu}
	
	\author{Wenjing Liu}
	\address{University of New Hampshire}
	\email{wbs4@wildcats.unh.edu}
	
	\author{Lauren Sager}
	\address{University of New Hampshire}
	\email{lbq32@wildcats.unh.edu}
	\thanks{}
	\dedicatory{Dedicated to Lyra, a great Chinese-American girl\\
	}

	\begin{abstract}
		In 2016 and 2017, Haihui Fan, Don Hadwin and Wenjing Liu proved a commutative and noncommutative version of Beurling's theorems for a continuous unitarily invariant norm $\alpha $ on $L^{\infty}(\mathbb{T},\mu)$ and tracial finite von Neumann algebras $\left( \mathcal{M},\tau \right) $, respectively. In the paper, we study unitarily $\|\|_{1}$-dominating invariant norms $\alpha $ on finite von Neumann algebras. First we get a Burling theorem in commutative von Neumann algebras by defining $H^{\alpha}(\mathbb{T},\mu)=\overline {H^{\infty}(\mathbb{T},\mu)}^{\sigma(L^{\alpha}\left(
			\mathbb{T} \right),\mathcal{L}^{\alpha^{'}}\left(
			\mathbb{T} \right))}\cap L^{\alpha}(\mathbb{T},\mu)$, then prove that the generalized Beurling theorem holds. Moreover, we get similar result in noncommutative case. The key ingredients in the proof of our result include a factorization theorem and a density theorem for $L^{\alpha }\left(\mathcal{M},\tau \right) $.
		
	\end{abstract}
	
	\maketitle

	\section{Introduction}
	
	Let $\mathbb{T}$ be the unit circle, i.e., $\mathbb{T}=\{ \lambda\in
	\mathbb{C}:\left\vert \lambda\right\vert =1\}$, and let $\mu$ be Haar measure
	(i.e., normalized arc length) on $\mathbb{T}$. The classical and influential
	Beurling-Helson-Lowdenslager theorem (see \cite{Beurling},\cite{helson},\cite{H.L}) states that if $W$ is a closed
	$H^{\infty}(\mathbb{T},\mu)$-invariant subspace (or, equivalently,
	$zW\subseteq W)$ of $L^{2}\left(  \mathbb{T},\mu\right)  \text{,}$ then
	$W=\varphi H^{2}$ for some $\varphi\in L^{\infty}(\mathbb{T},\mu)\text{,}$
	with $\left\vert \varphi\right\vert =1$ a.e.$(\mu)$ or $W=\chi_{E}%
	L^{2}(\mathbb{T},\mu)$ for some Borel set $E\subset\mathbb{T}$. If $0\neq
	W\subset H^{2}(\mathbb{T},\mu)$, then $W=\varphi H^{2}(\mathbb{T},\mu)\;$for
	some $\varphi\in H^{\infty}(\mathbb{T},\mu)$ with $\left\vert \varphi
	\right\vert =1$ a.e. $(\mu)$. Later, the Beurling's theorem was extended to
	$L^{p}(\mathbb{T},\mu)$ and $H^{p}(\mathbb{T},\mu)$ with $1\leq p\leq\infty$,
	with the assumption that $W$ is weak*-closed when $p=\infty$ (see
	\cite{halmos},\cite{helson},\cite{H.L},\cite{hoffman}). In \cite{chen2}, Yanni Chen extended the Helson-Lowdenslager-Beurling theorem for all continuous
	$\Vert \Vert_{1}$-dominating normalized gauge norms on $\mathbb{T}$. In \cite{F.H.W}, \cite{F.H.W-2} Haihui Fan, Don Hadwin and Wenjing Liu proved a commutative and noncommutative version of Beurling's theorems for a continuous unitarily invariant norm $\alpha $ on $L^{\infty}(\mathbb{T},\mu)$ and a tracial finite von Neumann algebra $\left( \mathcal{M},\tau \right) $, respectively. Later, Lauren Sager and Wenjing Liu got a similarly result for semifinite von Neumann algebras in \cite{sager-liu}.\\
	In this paper, we first extend the Helson-Lowdenslager-Beurling theorem for a much larger class of norms, $\Vert \Vert_{1}$-dominating normalized gauge norms on $L^{\infty} \left(\mathbb{T},\mu \right) $. For each such norm $\alpha $, we define the dual norm $\alpha{'}$, let $
	\mathcal{L}^{\alpha}(\mathbb{T},\mu)=\{f:f\; \text{is a measurable function on}\;
	\mathbb{T}\; \; \text{with}\; \alpha(f)<\infty\} \text{,}$ and
	$
	L^{\alpha}(\mathbb{T},\mu)=\overline{L^{\infty}(\mu)}^{\alpha}$, i.e.,
		the $\alpha $-closure of $L^{\infty}(\mu)$ in $
	\mathcal{L}^{\alpha}(\mathbb{T},\mu)\text{.}
	$ We have Banach space $L^{\alpha} \left(\mathbb{T},\mu \right)=\overline{L^{\infty} \left(\mathbb{T},\mu \right)}^{\alpha}$ and a Hardy space $H^{\alpha}=\overline{H^{\infty} \left(\mathbb{T},\mu \right)}^{\sigma(L^{\alpha}\left(
		\mathbb{T} \right),\mathcal{L}^{\alpha^{'}}\left(
		\mathbb{T} \right))}\cap L^{\alpha} \left(\mathbb{T},\mu \right) $with
	$L^{\infty}(\mathbb{T},\mu)\subset L^{\alpha}(\mathbb{T},\mu)\subset L^{1}(\mathbb{T},\mu)$ and $H^{\infty}(\mathbb{T},\mu)\subset H^{\alpha}(\mathbb{T},\mu)\subset H^{1}(\mathbb{T},\mu)$.
	In this new setting, we prove the following Beurling-Helson-Lowdenslager theorem, which is the main result of this paper.
	
	{\footnotesize THEOREM} \ref{sl-blc} Suppose $\mu$ is Haar measure on $\mathbb{T}$ and $\alpha$ is a normalized gauge norm on $L^{\infty} \left(
			\mathbb{T} \right) $ with $\alpha(.) \geq \Vert .\Vert_{1}$. 
			Let $M$ be an $\sigma(L^{\alpha}\left(
			\mathbb{T} \right),L^{\alpha^{'}}\left(
			\mathbb{T} \right))$-closed linear subspace of $L^{\alpha}\left(
			\mathbb{T} \right)$ with $zM\subseteq M$ if and only if either $W=\varphi
			H^{\alpha}(\mu)$ for some unimodular function $\varphi$, or $W=\chi
			_{E}L^{\alpha}(\mu)$, for some Borel subset $E$ of $\mathbb{T}$. If $0\neq
			W\subset H^{\alpha}(\mu)$, then $W=\varphi H^{\alpha}(\mu)$ for some inner
			function $\varphi$.
		
		To prove Theorem {\footnotesize THEOREM} \ref{sl-blc}. we need the following technical theorems in Section 3.
	    {\footnotesize THEOREM} \ref{factorization1} Let  $\alpha$ be a normalized gauge norm on $L^{\infty} \left(
			\mathbb{T} \right) $ with $\alpha(.) \geq \Vert .\Vert_{1}$. If $k \in L^{\infty}$, $k^{ -1} \in$ $L^{\alpha}$, then
			there is a unimodular function $u \in L^{\infty}$ and an outer function $s \in
			H^{\infty}$ such that $k =us$ and $s^{ -1} \in$ $H^{\alpha}$.\\
	 {\footnotesize THEOREM} \ref{density1}
				Suppose $\alpha$ is a normalized gauge norm on $L^{\infty} \left(
				\mathbb{T} \right) $ with $\alpha(.) \geq \Vert .\Vert_{1}$.
				Let $M$ be an $\sigma(L^{\alpha}\left(
				\mathbb{T} \right),L^{\alpha^{'}}\left(
				\mathbb{T} \right))$-closed linear subspace of $L^{\alpha}\left(
				\mathbb{T} \right)$ with $zM\subseteq M$. Then  
				\newline(1) $M\cap L^{\infty}\left(
				\mathbb{T} \right)$ is
				weak*-closed in $L^{\infty}\left(
				\mathbb{T} \right)$, 
				\newline(2) $M=\overline{M\cap
					L^{\infty}\left(
					\mathbb{T} \right)}^{\sigma(L^{\alpha}\left(
					\mathbb{T} \right),L^{\alpha^{'}}\left(
					\mathbb{T} \right))}\text{.}$
		
	 In noncommutative case,  we get similarly result. Suppose $\mathcal{M}$ is a finite von Neumann algebra with a faithful, normal,
	 tracial state $\tau $, $\Phi _{\tau }$ be the conditional expectation and $%
	 \alpha $ is a normalized, unitarily invariant $\|\|_{1}$-dominating norm
	 on $\mathcal{M}$. Let $L^{\alpha }(\mathcal{M},\tau )$ be the $\alpha $
	 closure of $\mathcal{M}$,i.e., $L^{\alpha }(\mathcal{M},\tau )=[\mathcal{M}%
	 ]_{\alpha }$. Similarly, $H^{\alpha }(\mathcal{M},\tau )= \overline{H^{\infty }(%
	 	\mathcal{M},\tau )}^{\sigma(L^{\alpha}\left(
	 	\mathcal{M},\tau \right),L^{\alpha^{'}}\left(
	 	\mathcal{M},\tau \right))}\cap L^{\alpha }(\mathcal{M},\tau ) $, $H_{0}^{\infty }(\mathcal{M},\tau )=\ker
	 (\Phi _{\tau })\cap H^{\infty }(\mathcal{M},\tau )$ and $H_{0}^{\alpha }(%
	 \mathcal{M},\tau )=\ker (\Phi _{\tau })\cap H^{\alpha }(\mathcal{M},\tau ).$ Then we get the following generalized Beurling theorem in finite von Neumann algebras.
	 
	  {\footnotesize THEOREM} \ref{main} Let $\mathcal{M}$ be a finite von Neumann algebra with a faithful, normal, tracial state $\tau $ and $\alpha $ be a 
	 	normalized, unitarily invariant, $\|\|_{1}$-dominating norm on $\mathcal{M}$.Let $H^{\infty }$
	 	be a finite subdiagonal subalgebra of $\mathcal{M}$ and $\mathcal{D}%
	 	=H^{\infty }\cap (H^{\infty })^{\ast }$. If $\mathcal{W}$ is a closed
	 	subspace of $L^{\alpha }(\mathcal{M},\tau )$ such that $\mathcal{W}H^{\infty
	 	}\subseteq \mathcal{W}$, then there exists a $\sigma(L^{\alpha}\left(
	 	\mathcal{M},\tau \right),L^{\alpha^{'}}\left(
	 	\mathcal{M},\tau \right))$ closed subspace $\mathcal{Y}$
	 	of $L^{\alpha }(\mathcal{M},\tau )$ and a family $\{u_{\lambda }\}_{\lambda
	 		\in \Lambda }$ of partial isometries in $\mathcal{M}$ such that\newline
	 	(1) $u_{\lambda }^{\ast }\mathcal{Y}=0$ for all $\lambda \in \Lambda $,%
	 	\newline
	 	(2) $u_{\lambda }^{\ast }u_{\lambda }\in \mathcal{D}$ and $u_{\lambda
	 	}^{\ast }u_{\mu }=0$ for all $\lambda ,\mu \in \Lambda $ with $\lambda \neq
	 	\mu $,\newline
	 	(3) $\mathcal{Y}=\overline{H_{0}^{\infty }\mathcal{Y}}^{\sigma(L^{\alpha}\left(
	 		\mathcal{M},\tau \right),L^{\alpha^{'}}\left(
	 		\mathcal{M},\tau \right))}$,\newline
	 	(4) $\mathcal{W}=\mathcal{Y}\oplus ^{col}(\oplus _{\lambda \in
	 		\Lambda }^{col}u_{\lambda }H^{\alpha }).$\\
	The organization of the paper is as follows. In section 2, we introduce $\|\|_{1}$-dominating normalized, unitarily invariant norms. In section 3, we study the relations between commutative Hardy spaces $H^{\alpha }(\mathbb{T},\mu)$ and get the generalized Beurling theorem in the commutative von Neumann algebras setting. In
	section 4, using similar techniques as in section 3, we prove a version of the generalized noncommutative Beurling's theorem for finite von Neumann algebras.

\section{Gauge Norms on the Unit Circle}
 A norm $\alpha$ on $L^{\infty} (\mathbb{T},\mu)$ is a \emph{normalized gauge norm} if
\begin{enumerate}
	\item $\alpha(1) =1\text{,}$
	
	\item $\alpha(\vert f\vert) =\alpha(f)$ for every $f \in L^{\infty}
	(\mathbb{T},\mu)$.
\end{enumerate}

We say that a\textbf{\ }\emph{normalized gauge norm} $\alpha$ is $ \Vert
\cdot\Vert_{1 ,\mu}$\emph{-dominating}  if there exists $c\in R^{+}$ such that \text{(3)} $\alpha(f) \geq c \Vert f\Vert_{1 ,\mu} ,\; \text{for every}\;f \in L^{\infty}
(\mathbb{T},\mu)\text{.}$

For example, it is easily to see the following fact that
\begin{enumerate}
	\item The common norm $\Vert\cdot\Vert_{p,\mu}$ is a $\alpha$ norm for $1\leq p\leq\infty$.
    \item If $1\leq p_{n}<\infty$ for $n\geq 1$, $\sum_{n=1}^{\infty}\frac{1}{2^{n}}\Vert\cdot\Vert_{p_{n},\mu}$ is a $\alpha$ norm, which is not equivalent to any $\Vert\cdot\Vert_{p,\mu}$.
\end{enumerate}

We can extend the normalized gauge norm $\alpha$ from $L^{\infty}(\mathbb{T},\mu)$
to the set of all measurable functions, and define $\alpha$ for all measurable
functions $f$ on $\mathbb{T}$ by%
\[
\alpha(f)=\sup\{ \alpha(s):s\; \text{is a simple function}\;,0\leq s\leq|f|\}
\text{.}%
\]
It is clear that $\alpha(f)=\alpha(|f|)$ still holds. \newline Define the following two spaces.
\[
\mathcal{L}^{\alpha}(\mathbb{T},\mu)=\{f:f\; \text{is a measurable function on}\;
\mathbb{T}\; \; \text{with}\; \alpha(f)<\infty\} \text{,}%
\]%
\[
L^{\alpha}(\mathbb{T},\mu)=\overline{L^{\infty}(\mu)}^{\alpha},\; \text{i.e.,
	the}\; \alpha\; \text{-closure of}\;L^{\infty}(\mu)\; \text{in}\;
\mathcal{L}^{\alpha}(\mathbb{T},\mu)\text{.}%
\]

The following are some properties of $\alpha$ norm in []

\begin{lemma}
	 Suppose $f,g: \mathbb{T}\rightarrow \mathbb{C}$ are measurable. Let $\alpha$ be a $\left\Vert \cdot \right\Vert _{1,\mu}$-dominating normalized gauge norm. Then the following statements are true\\
	 (1) If $|f|\leq |g|$, then $\alpha(f)\leq \alpha(g)$;\\
	 (2)  $\alpha(fg)\leq \alpha(f)\left\Vert g \right\Vert _{\infty}$;\\
	 (3) $\alpha(g)\leq \left\Vert g \right\Vert _{\infty}$;\\
	 (4) $L^{\infty} (\mathbb{T},\mu)\subset L^{\alpha} (\mathbb{T},\mu)\subset\mathcal{L}^{\alpha} (\mathbb{T},\mu)\subset L^{1} (\mathbb{T},\mu)$ .
\end{lemma}

Let $\alpha$ be a $\left\Vert \cdot \right\Vert _{1,\mu}$-dominating normalized gauge norm on $L^{\infty} \left(
\mathbb{T},\mu\right) $. We define the dual norm $\alpha^{'}: L^{\infty} (\mathbb{T},\mu)\rightarrow [0,\infty] $ by
\[ \alpha^{'}(f)=sup\{|\int_{\mathbb{T}}fhd\mu|:h \in L^{\infty} (\mathbb{T},\mu),\alpha(h)\leq 1\}\]
\[  =sup\{\int_{\mathbb{T}}|fh|d\mu:h \in L^{\infty} (\mathbb{T},\mu),\alpha(h)\leq 1\} \]

\begin{lemma}
Let $\alpha$ be a $\Vert \Vert_{1}$-dominating normalized gauge norm on $L^{\infty} \left(
\mathbb{T},\mu \right) $. Then the dual norm $\alpha^{'}$ is also a $\Vert \Vert_{1}$-dominating normalized gauge norm on $L^{\infty} \left(
\mathbb{T},\mu \right) $.
\end{lemma}

We also can define the dual spaces of $\mathcal{L}^{\alpha}(\mathbb{T},\mu)$ and $L^{\alpha}(\mathbb{T},\mu)$.
\[
\mathcal{L}^{\alpha^{'}}(\mathbb{T},\mu)=\{f:f\; \text{is a measurable function on}\;
\mathbb{T}\; \; \text{with}\; \alpha^{'}(f)<\infty\} \text{.}%
\]

\[
L^{\alpha^{'}}(\mathbb{T},\mu)=\overline{L^{\infty}(\mu)}^{\alpha^{'}},\; \text{i.e.,
	the}\; \alpha^{'}\; \text{-closure of}\;L^{\infty}(\mathbb{T},\mu)\; \text{in}\;
\mathcal{L}^{\alpha^{'}}(\mathbb{T},\mu)\text{.}%
\]\\
By lemma 2.3 in \cite{chen2}, we have 
\[L^{\infty} (\mathbb{T},\mu)\subset L^{\alpha^{'}} (\mathbb{T},\mu)\subset\mathcal{L}^{\alpha^{'}} (\mathbb{T},\mu)\subset L^{1} (\mathbb{T},\mu)\]\\
Now we consider the $\sigma(L^{\alpha}\left(
\mathbb{T} \right),\mathcal{L}^{\alpha^{'}}\left(
\mathbb{T} \right))$ topology on $L^{\alpha} (\mathbb{T},\mu)$ space. Since $L^{\infty} (\mathbb{T},\mu)\subset L^{\alpha} (\mathbb{T},\mu)$, $\overline {L^{\infty} (\mathbb{T},\mu)}^{\sigma(L^{\alpha}\left(
	\mathbb{T} \right),\mathcal{L}^{\alpha^{'}}\left(
	\mathbb{T} \right))}\subset \overline{L^{\alpha} (\mathbb{T},\mu)}^{\sigma(L^{\alpha}\left(
	\mathbb{T} \right),\mathcal{L}^{\alpha^{'}}\left(
	\mathbb{T} \right))}=L^{\alpha} (\mathbb{T},\mu)$. Thus we have the following result.
\begin{lemma}
	$\overline {L^{\infty} (\mathbb{T},\mu)}^{\sigma(L^{\alpha}\left(
		\mathbb{T} \right),\mathcal{L}^{\alpha^{'}}\left(
		\mathbb{T} \right))}=L^{\alpha} (\mathbb{T},\mu)$.
\end{lemma}
\begin{proof}
	As we show above, $\overline {L^{\infty} (\mathbb{T},\mu)}^{\sigma(L^{\alpha}\left(
		\mathbb{T} \right),\mathcal{L}^{\alpha^{'}}\left(
		\mathbb{T} \right))}\subset L^{\alpha} (\mathbb{T},\mu)$. Additionally, by properties of norm and weak closure, we have $L^{\alpha} (\mathbb{T},\mu) \subset \overline {L^{\infty} (\mathbb{T},\mu)}^{\sigma(L^{\alpha}\left(
		\mathbb{T} \right),\mathcal{L}^{\alpha^{'}}\left(
		\mathbb{T} \right))}$.
\end{proof}
Since $L^{\infty} \left(  \mathbb{T},\mu\right)  $ with the norm $\alpha$ is dense
in $L^{\alpha} (\mathbb{T},\mu)$, they have the same dual spaces. i.e. the normed dual $\left(  L^{\alpha} (\mathbb{T},\mu)
,\alpha\right)  ^{\#} =\left(  L^{\infty} \left(  \mathbb{T},\mu\right)
,\alpha\right)  ^{\#}$. By the following lemma, we can view the dual space as a vector space, a vector subspace of $L^{1}
(\mathbb{T},\mu)$. 
Suppose $w \in L^{1} (\mathbb{T},\mu)$, we define the functional
$\varphi_{w} :L^{\infty} (\mathbb{T},\mu) \rightarrow\mathbb{C}$ by%
\[
\varphi_{w} \left(  f\right)  =\int_{\mathbb{T}}fw d\mu\text{.}%
\]

\begin{lemma}
	\label{dual} Let $\alpha$ be a $\Vert \cdot \Vert_{1}$-dominating normalized gauge norm on $L^{\infty} \left(
	\mathbb{T},\mu\right) $ and  $\alpha^{'}$ be its dual norm . Then \newline(1) if $\varphi:L^{\infty}(\mathbb{T},\mu
	)\rightarrow\mathbb{C}$ is an $\sigma(L^{\alpha}\left(
	\mathbb{T} \right),\mathcal{L}^{\alpha^{'}}\left(
	\mathbb{T} \right))$-continuous linear functional, then
	there is a $w\in L^{1}(\mathbb{T},\mu)$ such that $\varphi=\varphi_{w}$, where  $\varphi_{w} \left(  f\right)  =\int_{\mathbb{T}}fw d\mu\text{.}$
	\newline(2) if $\varphi_{w}$ is $\sigma(L^{\alpha}\left(
	\mathbb{T} \right),\mathcal{L}^{\alpha^{'}}\left(
	\mathbb{T} \right))$-continuous on $L^{\infty}(\mathbb{T}
	,\mu)$, then \newline(a) $\left\Vert w\right\Vert _{1,\mu}\leq\left\Vert
	\varphi_{w}\right\Vert =\left\Vert \varphi_{\left\vert w\right\vert
	}\right\Vert \text{,}$ \newline(b) given $\varphi\;$in the dual of $L^{\alpha
}(\mathbb{T},\mu)$, i.e., $\varphi\in\left(  L^{\alpha}(\mathbb{T},\mu
),\sigma(L^{\alpha}\left(
\mathbb{T} \right),\mathcal{L}^{\alpha^{'}}\left(
\mathbb{T} \right))\right)^{\#}\text{,}$ there exists a $w\in L^{1}(\mathbb{T},\mu)$, such
that \newline%
\[
\forall f\in L^{\infty}(\mathbb{T},\mu),\varphi(f)=\int_{\mathbb{T}}%
fwd\mu \; \text{and}\;wL^{\alpha}(\mathbb{T},\mu)\subseteq L^{1}%
(\mathbb{T},\mu)%
\]
\newline(3) we have $L^{\alpha^{'}} (\mathbb{T},\mu)\subseteq \left(  L^{\alpha}(\mathbb{T},\mu
),\sigma(L^{\alpha}\left(
\mathbb{T} \right),\mathcal{L}^{\alpha^{'}}\left(
\mathbb{T} \right))\right)^{\#}\text{.}$

\end{lemma}

\begin{proof}
	For (1), It's easy to check by the definition of $\sigma(L^{\alpha}\left(
	\mathbb{T} \right),\mathcal{L}^{\alpha^{'}}\left(
	\mathbb{T} \right))$-continuous linear functional\\
	For (2a), From (1), we have
	\[
	\left\Vert w\right\Vert _{1,\mu}=\sup\left\{  \left\vert \int_{\mathbb{T}}%
	wsd\mu\right\vert :s\; \text{is simple,}\; \left\Vert s\right\Vert _{\infty
	}\leq1\right\}
	\]%
	\[
	=\sup\left\{  \left\vert \varphi\left(  s\right)  \right\vert :s\;
	\text{simple,}\; \left\Vert s\right\Vert _{\infty}\leq1\right\}
	\leq\left\Vert \varphi\right\Vert \text{.}%
	\]
 We will see $\left\Vert w\right\Vert _{1,\mu}\leq\left\Vert
	\varphi\right\Vert .$
	
	(2b) Suppose $f\in L^{\alpha}(\mathbb{T},\mu)$, $f=u|f|,|u|=1$. $|f|\in L^{\alpha
}(\mathbb{T},\mu)$. There exists an increasing positive sequence $s_{n}$ such
that $s_{n}\rightarrow|f|$ a.e. $(\mu)$, thus $us_{n}\rightarrow u|f|$
a.e.$(\mu)$. $\forall w\in L^{1}(\mathbb{T},\mu),w=v|w|$, where $|v|=1$, so we
have $\overline{v}s_{n}\rightarrow\overline{v}|f|$ a.e. $(\mu)$, where
$\overline{v}$ is the conjugate of $v$ and $\alpha(\overline{v}s_{n}%
-\overline{v}|f|)\rightarrow0$. Thus we have $\varphi(\overline{v}s_{n})\rightarrow
\varphi(\overline{v}|f|)$. On the other hand, we also have $\varphi
(\overline{v}s_{n})=\int_{\mathbb{T}}\overline{v}s_{n}wd\mu\rightarrow
\int_{\mathbb{T}}\overline{v}|f|wd\mu=\int_{\mathbb{T}}|f||w|d\mu$ by monotone
convergence theorem. Thus $\int_{\mathbb{T}}|f||w|d\mu=\int_{\mathbb{T}
}|f|\overline{v}wd\mu=\varphi(\overline{v}|f|)<\infty$, therefore $fw\in
L^{1}(\mathbb{T},\mu)$, i.e., $wL^{\alpha}(\mathbb{T},\mu)\subseteq
L^{1}(\mathbb{T},\mu)$, where $w\in L^{1}(\mathbb{T},\mu)$.\\
For (3), By (2b) we know that if $\varphi \in L^{\alpha}(\mathbb{T},\mu) $, then there exists $w \in L^{1} (\mathbb{T},\mu)$ such that $\varphi(f)=\varphi_{w}(f), \forall f\in L^{\infty}(\mathbb{T},\mu)$. By (1),
$\varphi(f)=\varphi_{w}(f)$ implies $\varphi$ is an $\sigma(L^{\alpha}\left(
\mathbb{T} \right),\mathcal{L}^{\alpha^{'}}\left(
\mathbb{T} \right))$-continuous linear functional.

\end{proof}

\section{The Extension of Beurling Theorem in Communtative von Neumann Algebras }
Let $\alpha$ be a $\Vert \cdot \Vert_{1}$-dominating normalized gauge norm on $L^{\infty} \left(
\mathbb{T},\mu\right) $. We define $H^{\alpha}(\mathbb{T},\mu)=\overline {H^{\infty}(\mathbb{T},\mu)}^{\sigma(L^{\alpha}\left(
	\mathbb{T} \right),\mathcal{L}^{\alpha^{'}}\left(
	\mathbb{T} \right))}\cap L^{\alpha}(\mathbb{T},\mu)$, from the definition, we first extend the classical $L^{p}(\mathbb{T},\mu)$ spaces.

\begin{example}
	If we take $\alpha$ to be $p$-norm, then $H^{p}(\mathbb{T},\mu)=\overline {H^{\infty}(\mathbb{T},\mu)}^{\sigma(L^{p}\left(
		\mathbb{T} \right),\mathcal{L}^{q}\left(
		\mathbb{T} \right))} \cap L^{p}(\mathbb{T},\mu)$.
\end{example}

In addition, in the classical Hardy space, we have  $H^{p}(\mathbb{T},\mu) =H^{1} (\mathbb{T},\mu) \cap L^{p}(\mathbb{T},\mu)\text{,}$ now we have similar result in the following theorem.
	\begin{theorem}
		\label{N8} $H^{\alpha}(\mathbb{T},\mu) =H^{1} (\mathbb{T},\mu) \cap L^{\alpha}(\mathbb{T},\mu)\text{.}$
	\end{theorem}
	
	\begin{proof}
		
		By the definition of $	H^{\alpha} $, we know that
		\[
		H^{\alpha} =\overline{H^{\infty} (\mu)}^{\sigma(L^{\alpha}\left(
			\mathbb{T} \right),L^{\alpha^{'}}\left(
			\mathbb{T} \right))} \cap L^{\alpha}(\mathbb{T},\mu) \subset\overline
		{L^{\infty}}^{\sigma(L^{\alpha}\left(
			\mathbb{T} \right),L^{\alpha^{'}}\left(
			\mathbb{T} \right))} = L^{\alpha}(\mathbb{T},\mu)\text{.}%
		\]

	For every $f\in H^{\alpha} =\overline{H^{\infty} (\mu)}^{\sigma(L^{\alpha}\left(
		\mathbb{T} \right),L^{\alpha^{'}}\left(
		\mathbb{T} \right))} \cap L^{\alpha}(\mathbb{T},\mu)\subsetneq L^{1}(\mathbb{T},\mu) $, there is a sequence ${f_{n}}$ in $H^{\infty}$ such that  $f_{n}\rightarrow f $ in $\sigma(L^{\alpha}\left(
	\mathbb{T} \right),L^{\alpha^{'}}\left(
	\mathbb{T} \right)) $ topology. Thus, for every $g\in L^{\alpha^{'}\left(
		\mathbb{T} \right)}$,
	$\int_{\mathbb{T}} (f_{n}g) d\mu
	\rightarrow \int_{\mathbb{T}}(fg)d\mu$.
	Therefore, we have
	
	\[c_{-m}=\int_{\mathbb{T}}fz^{m}d\mu=\lim_{n\rightarrow \infty} \int_{\mathbb{T}}f_{n}z^{m}d\mu=\lim_{n\rightarrow \infty}c_{-mn}=\lim_{n\rightarrow \infty}0=0, m\geq 0\]
	
	So $f\in  H^{1}(\mathbb{T},\mu) $. Thus $H^{\alpha}(\mathbb{T},\mu) \subseteq H^{1}(\mathbb{T},\mu) \cap L^{\alpha}(\mathbb{T},\mu)\text{.}$\\
	
    Now since $H^{\alpha}(\mathbb{T},\mu)$ is an $\sigma(L^{\alpha}\left(
    \mathbb{T} \right),L^{\alpha^{'}}\left(
    \mathbb{T} \right))$-closed subspace of $L^{\alpha}(\mathbb{T},\mu)$, for every $f \in L^{\alpha}(\mathbb{T},\mu)$ and $f \notin H^{\alpha}(\mathbb{T},\mu)$, there is a $\sigma(L^{\alpha}\left(
    \mathbb{T} \right),L^{\alpha^{'}}\left(
    \mathbb{T} \right))$-continuous functional $\varphi$ on $L^{\alpha}(\mathbb{T},\mu)$ such that
    $\varphi(H^{\alpha}(\mathbb{T},\mu))=0$ and $\varphi(f)\neq 0$. Also, there is a $g \in L^{\alpha^{'}}(\mathbb{T},\mu)$ such that $\varphi(h)=\int_{\mathbb{T}}hgd\mu$ for all $h\in L^{\alpha}(\mathbb{T},\mu)$ . And we know   $g \in L^{\alpha^{'}}(\mathbb{T},\mu)\subset L^{1}(\mathbb{T},\mu) $, so we can write $g(z)=\sum_{n=-\infty}^{\infty}c_{n}z^{n}$. Since $\phi(H^{\alpha}(\mathbb{T},\mu))=0$,we have
    \[c_{-n}=\int_{\mathbb{T}}gz^{n}d\mu=\varphi(z^{n})=0, n\geq0 .\]
    Thus $g$ is analytic and $g(0)=0$.\\
    Take $w\in H^{1} (\mathbb{T},\mu) \cap L^{\alpha}(\mathbb{T},\mu)$,then $wg$ is analytic and $wg\in L^{1}(\mathbb{T},\mu)$. Hence
    \[\varphi(w)=\int_{\mathbb{T}}wgd\mu=w(0)g(0)=0\]
    Since for every $f \in L^{\alpha}(\mathbb{T},\mu)$ and $f \notin H^{\alpha}(\mathbb{T},\mu)$, there is a $\varphi$ is $\sigma(L^{\alpha}\left(
    \mathbb{T} \right),L^{\alpha^{'}}\left(
    \mathbb{T} \right))$-continuous functional on $L^{\alpha}(\mathbb{T},\mu)$ such that
    $\varphi(H^{\alpha}(\mathbb{T},\mu))=0$ and $\varphi(f)\neq 0$, $w\in H^{\alpha}(\mathbb{T},\mu) $ by Hahn-Banach theorem, which implies $H^{1} (\mathbb{T},\mu) \cap L^{\alpha}(\mathbb{T},\mu)\subset H^{\alpha}(\mathbb{T},\mu)$.
    	
	\end{proof}

	\begin{lemma} \label{including} Let $\alpha$ be a $\Vert \cdot \Vert_{1}$-dominating normalized gauge norm on $L^{\infty} \left(
		\mathbb{T},\mu\right) $. If $W$ is an $\sigma(L^{\alpha}\left(
		\mathbb{T} \right),L^{\alpha^{'}}\left(
		\mathbb{T} \right))$-closed linear subspace of $L^{\alpha}(\mathbb{T},\mu)$ with $zW\subseteq W$, then $H^{\infty}(\mu)W\subset W$.
		
	\end{lemma}
	
	\begin{proof}
	
	Let $P^{+}=\{ e_n: n\in \mathbb{N}\}$ denote the class of all polynomials in $H^{\infty}(\mathbb{T},\mu)$, where $e_{n}(z)=z^{n}$ for all $z$ in the unit circle $\mathbb{T}$. Since $zW\subseteq W$ , we see $p(z)W\subseteq W$ for any polynomial $p\in P^{+}$. To complete the proof, it suffices to show that $fh\in W$ for every $h\in W$and every $f\in H^{\infty}(\mathbb{T},\mu)$. Now we assume that $u$ is a nonzero element in $L^{\alpha^{'}}(\mathbb{T},\mu)$, then it follows from lemma \ref{dual} (2b)
    $hu\in WL^{\alpha^{'}}(\mathbb{T},\mu)\subset L^{\alpha}\left(
    \mathbb{T} \right)L^{\alpha^{'}}\left(
    \mathbb{T} \right)\subset L^{1}\left(
    \mathbb{T} \right)$. Since $f\in H^{\infty}$, define $\varphi(h)=\int_{\mathbb{T}}hgd\mu$ for all $h\in L^{\alpha}(\mathbb{T},\mu)$ , now we have $c_{-n}=\int_{\mathbb{T}}fz^{n}d\mu=\varphi(z^{n})=0$, for all $n>0$, which implies that the partial sums $S_{n}(f)=\sum_{-n}^{n}c_{n}e^{n}=\sum_{0}^{n}c_{n}e^{n}\in P+$ for all $n>0$.
    Hence the Cesaro means 
    \[\sigma_{n}(f)=\frac{S_{0}(f)+S_{1}(f)+...+S_{n}(f)}{n+1}\in P^{+} \]
    Moreover, we know that $\sigma_{n}(f)\rightarrow f$ in the weak* topology. Since $hu\in L^{1}\left(
   \mathbb{T} \right)$ we have
   \[\int_{\mathbb{T}}\sigma_{n}(f)hu d\mu \rightarrow \int_{\mathbb{T}}fhu d\mu \]
   Observe that $\sigma_{n}(f)h\in P^{+}W\subset W$ and $u\in L^{\alpha^{'}}\left(
   \mathbb{T} \right)$, it follows that $\sigma_{n}(f)h\rightarrow fh$ in $\sigma(L^{\alpha}\left(
   \mathbb{T} \right),L^{\alpha^{'}}\left(
   \mathbb{T} \right))$ topology. Since $W$ is $\sigma(L^{\alpha}\left(
   \mathbb{T} \right),L^{\alpha^{'}}\left(
   \mathbb{T} \right))$-closed, $fh\in W$. This completes the proof.
    \end{proof}	
	
    A key ingredient is based on the following result that uses the Herglotz kernel in \cite{chen1}.
	\begin{lemma}
		$ \{|h|: 0\neq h\in H^{1}(\mathbb{T},\mu)\}=\{f\in L^{1}(\mathbb{T},\mu): f\geq 0$ and $ \log f \in L^{1}(\mathbb{T},\mu) \} $ , In fact, if $f \geq 0$ and $\phi, \log \phi \in L^{1}(\mathbb{T},\mu)$, then \\
	\[ f(z)=exp{\int_{\mathbb{T}}\frac{w+z}{w-z}\log f(w) d\mu(w) }\]
	
	defines an outer function $h$ on $\mathbb{D}$ and $|h|=f$ on $\mathbb{T}$.
  \end{lemma}
	
	The following result is a factorization theorem for $L^{\alpha}(\mathbb{T},\mu)\text{.}$
	
	\begin{theorem}
		\label{factorization1} Let $\alpha$ be a $\Vert \cdot \Vert_{1}$-dominating normalized gauge norm on $L^{\infty} \left(
		\mathbb{T},\mu\right) $. If $k \in L^{\infty}(\mathbb{T},\mu)$, $k^{ -1} \in$ $L^{\alpha}(\mathbb{T},\mu)$, then
		there is a unimodular function $u \in L^{\infty}(\mathbb{T},\mu)$ and an outer function $s \in
		H^{\infty}(\mathbb{T},\mu)$ such that $k =us$ and $s^{ -1} \in$ $H^{\alpha}(\mathbb{T},\mu)$.
	\end{theorem}
	
	\begin{proof}
		Recall that an outer function is uniquely determined by its absolute boundary
		values, which are necessarily absolutely log integrable. 
		Suppose  $k \in L^{\infty}(\mathbb{T},\mu)$, $k^{ -1} \in$ $L^{\alpha}(\mathbb{T},\mu)$, on the circle we have
		
		\[-\left\vert k\right\vert<-\log|k|= \log|k^{ -1}|\leq |k^{ -1}| \]
		
		It follows from $k \in L^{\infty}(\mathbb{T},\mu)$, and $k^{ -1} \in L^{\alpha}\left(
		\mathbb{T} \right)\subset L^{1}\left(
		\mathbb{T} \right)$ that
		
			\[-\infty < \int_{\mathbb{T}}-|k| d\mu\leq\int_{\mathbb{T}} \log|k^{ -1}| d\mu \leq \int_{\mathbb{T}} |k^{ -1}| d\mu < \infty \]
		
		Hence $|k^{ -1}|$ is log integrable, by lemma???,there is an outer function $h\in H^{1} (\mathbb{T},\mu)  $ such that $|h|=|k^{ -1}|$
		on $\mathbb{T}$. If we let $s=h^{-1}$ and $u=kh$, we know $h$ is outer, $s=h^{-1}$ is analytic on $\mathbb{D}$, also, $|s|=|h^{-1}|=|k|\in L^{\infty}$, so  $s \in
		H^{\infty}$ such that $k =us$ where $u$ is unimodular. Further, since
		$h\in H^{1} \left(  \mu\right)  $ and $uk^{-1}\in L^{\alpha}(\mathbb{T},\mu) $ , it follows from ??? that $s^{-1}=h=uk^{-1}\in H^{1} (\mathbb{T},\mu)
		\cap L^{\alpha}(\mathbb{T},\mu) =H^{\alpha}(\mathbb{T},\mu)$.
		
	\end{proof}

	We let $\mathbb{B} =\{f \in L^{\infty}(\mathbb{T},\mu) :\Vert f\Vert_{\infty} \leq1\}$ denote
	the closed unit ball in $L^{\infty} (\mathbb{T},\mu)$.
	
	\begin{lemma}
		\label{coincide}Let $\alpha$ be a $\Vert \cdot \Vert_{1}$-dominating normalized gauge norm on $L^{\infty} \left(
		\mathbb{T},\mu\right) $.
		Then $\mathbb{B}=\{f\in L^{\infty}(\mu):\Vert
		f\Vert_{\infty}\leq1\}$ is $\alpha$-closed.
	\end{lemma}
	
	\begin{proof}
		
		Suppose $\{f_{n}\}$ is a sequence in $\mathbb{B}$, $f\; \in
		L^{\alpha}$ and $\alpha(f_{n} -f) \rightarrow0$. Since $\Vert f\Vert_{1} \leq \alpha(f)$. it follows that $\Vert f_{n} -f\Vert_{1} \rightarrow 0$, which implies that $f_{n} \rightarrow f$ in $\mu
		$-measure. Then there is a subsequence $\{f_{n_{k}}\}$ such that $f_{n_{k}}
		\rightarrow f$ a.e. $\left(  \mu \right)  $. Hence $f \in\mathbb{B}%
		\text{.}$
	\end{proof}

	The following theorem and its corollary relate the closed invariant subspaces
	of $L^{\alpha}(\mathbb{T},\mu)$ to the weak*-closed invariant subspaces of $L^{\infty}$.
	
	The following lemma is the Krein-Smulian theorem from \cite{conway}.
	
	\begin{lemma}
		Let $X$ be a Banach space. A convex set in $X^{\#}$ is weak* closed if and only if its intersection with $\mathbb{B}=\{\phi: \Vert \phi \Vert \leq 1 \}$ is weak* closed.
	\end{lemma}

	\begin{theorem}
		\label {density1} Let $\alpha$ be a $\Vert \cdot \Vert_{1}$-dominating normalized gauge norm on $L^{\infty} \left(
		\mathbb{T},\mu\right) $.
		Let $M$ be an $\sigma(L^{\alpha}\left(
		\mathbb{T} \right),L^{\alpha^{'}}\left(
		\mathbb{T} \right))$-closed linear subspace of $ L^{\alpha}(\mathbb{T},\mu) $ with $zM\subseteq M$. Then  
		\newline(1) $M \cap L^{\infty}\left(
		\mathbb{T} \right)$ is
		weak*-closed in $L^{\infty}\left(
		\mathbb{T} \right)$, 
		\newline(2) $M=\overline{M\cap
			L^{\infty}\left(
			\mathbb{T} \right)}^{\sigma(L^{\alpha}\left(
			\mathbb{T} \right),L^{\alpha^{'}}\left(
			\mathbb{T} \right))}\text{.}$
	\end{theorem}
	
	\begin{proof}
	
		For (1), to prove $M\cap L^{\infty}(\mathbb{T},\mu)$ is weak*-closed in $L^{\infty
		}(\mathbb{T},\mu)$, using the Krein-Smulian theorem, we only need to show that $M\cap
		L^{\infty}(\mathbb{T},\mu)\cap\mathbb{B}$, i.e., $M\cap\mathbb{B}$, is weak*-closed.
		If $\{f_{\lambda }\}$ is a net in
		$M\cap\mathbb{B}$ and $f_{\lambda}\rightarrow f$ weak* in $L^{\infty}(\mathbb{T},\mu)$, then, for every $g\in L^{1}(\mathbb{T},\mu),\int_{\mathbb{T}}(f_{\lambda
		}-f)gd\mu\rightarrow0$. Since $\alpha^{'}\geq\Vert cdot{.}\Vert_{1}$, $\mathbb{L}{\alpha^{'}}(\mathbb{T},\mu)\subset L^{1}(\mathbb{T},\mu)$ and we have
		$f_{\lambda}\rightarrow f$in $ \sigma(L^{\alpha}\left(
		\mathbb{T} \right),L^{\alpha^{'}}\left(
		\mathbb{T} \right))$ topology, so $f\in M
		$. Since $\mathbb{B}$ is weak* closed, $f \in \mathbb{B}$ ,thus $f\in W\cap\mathbb{B}$. Hence $M\cap\mathbb{B}$ is weak*-closed in $L^{\infty}(\mu)$.

	    For (2), since $M$ is $\sigma(L^{\alpha}\left(
		\mathbb{T} \right),L^{\alpha^{'}}\left(
		\mathbb{T} \right))$-closed linear subspace of $L^{\alpha}(\mathbb{T},\mu)$, it is clear
		that $M\supset\overline{M\cap L^{\infty}(\mu)}^{\sigma(L^{\alpha}\left(
			\mathbb{T} \right),L^{\alpha^{'}}\left(
			\mathbb{T} \right))}$. 
		 Suppose $f\in M$
		and let $k=\frac{1}{|f|+1}$. Then $k\in L^{\infty}(\mathbb{T},\mu)$, $k^{-1}\in$
		$L^{\alpha}(\mathbb{T},\mu)\text{.}$ It follows from theorem \ref{factorization1} that there is
		an $s\in H^{\infty}(\mathbb{T},\mu),s^{-1}\in$ $H^{\alpha}(\mathbb{T},\mu)$ and an unimodular function
		$u$ such that $k=us$, so $sf=\overline{u}kf=\overline{u}\frac{f}{|f|+1}\in
		L^{\infty}(\mathbb{T},\mu)$. There is a sequence $\{s_{n}\}$ in $H^{\infty}(\mathbb{T},\mu)$
		such that $s_{n}\rightarrow s^{-1}$ in $\sigma(L^{\alpha}\left(
		\mathbb{T} \right),L^{\alpha^{'}}\left(
		\mathbb{T} \right)) $ topology. For each $n\in\mathbb{N}$, it
		follows from lemma \ref{including} that $s_{n}sf\in H^{\infty}(\mu
		)H^{\infty}(\mu)M\subset M$ and $s_{n}sf\in H^{\infty}(\mu)L^{\infty
		}(\mu)\subset L^{\infty}(\mu)$, which implies that $\{s_{n}sf\}$ is a
		sequence in $M\cap L^{\infty}(\mu)$.
		For every $g\in L^{\alpha^{'}\left(
			\mathbb{T} \right)}$,
		$\int_{\mathbb{T}} (s_{n}sf-f)g d\mu
		=\int_{\mathbb{T}} (s_{n}-s^{-1})sfgd\mu$.
		Since $s_{n}\rightarrow s^{-1}$ in $\sigma(L^{\alpha}\left(
		\mathbb{T} \right),L^{\alpha^{'}}\left(
		\mathbb{T} \right)) $ topology, and $sfg\in L^{\alpha^{'}\left(
			\mathbb{T} \right)}$,
		$\int_{\mathbb{T}}(s_{n}-s^{-1})sfg d\mu\rightarrow 0$.
		Thus $f\in\overline{M\cap
			L^{\infty}(\mu)}^{\sigma(L^{\alpha}\left(
			\mathbb{T} \right),L^{\alpha^{'}}\left(
			\mathbb{T} \right))}$. 
		Therefore $M=\overline{M\cap L^{\infty
			}(\mu)}^{\sigma(L^{\alpha}\left(
			\mathbb{T} \right),L^{\alpha^{'}}\left(
			\mathbb{T} \right))}$.
	\end{proof}
	
		\begin{lemma}
			\label{w*}A weak*-closed linear subspace $M$ of $L^{\infty} (\mathbb{T},\mu)$
			satisfies $z M \subset M$ if and only if $M =\varphi H^{\infty} (\mathbb{T},\mu)$ for
			some unimodular function $\varphi$ or $M =\chi_{E} L^{\infty} (\mathbb{T},\mu)$, for
			some Borel subset $E$ of $\mathbb{T}$.
		\end{lemma}
	
	\begin{theorem}
	\label{sl-blc}	Suppose $\mu$ is Haar measure on $\mathbb{T}$ and Let $\alpha$ be a $\Vert \cdot \Vert_{1}$-dominating normalized gauge norm on $L^{\infty} \left(
		\mathbb{T},\mu\right) $. 
		 	Let $W$ be an $\sigma(L^{\alpha}\left(
		 	\mathbb{T} \right),L^{\alpha^{'}}\left(
		 	\mathbb{T} \right))$-closed linear subspace of $L^{\alpha}\left(
		 	\mathbb{T} \right)$ with $zW\subseteq W$ if and only if either $W=\varphi
		H^{\alpha}(\mu)$ for some unimodular function $\varphi$, or $W=\chi
		_{E}L^{\alpha}(\mu)$, for some Borel subset $E$ of $\mathbb{T}$. If $0\neq
		W\subset H^{\alpha}(\mu)$, then $W=\varphi H^{\alpha}(\mu)$ for some inner
		function $\varphi$.
	\end{theorem}
	
	\begin{proof}
	 Let $M =W \cap L^{\infty}(\mathbb{T},\mu) $,  it follows from
		the (1) in theorem \ref{density1} that $M$ is weak* closed in $L^{\infty}
		(\mathbb{T},\mu)$. Since $z W \subset W$, it is easy to check that $z M \subset M$.
		Then by lemma \ref{w*}, we can conclude that either $M =\varphi
		H^{\infty}(\mathbb{T},\mu)$ for some unimodular function $\varphi$ or $M =\chi_{E}
		L^{\infty} (\mathbb{T},\mu)$, for some Borel subset $E$ of $\mathbb{T}$. By the (2)
		in theorem \ref{density1}, if $M =\varphi H^{\infty} (\mu)\text{,}$ $W
		=\overline{W \cap L^{\infty} (\mu)}^{\sigma(L^{\alpha}\left(
			\mathbb{T} \right),L^{\alpha^{'}}\left(
			\mathbb{T} \right))} 
		=\overline{M}^{\sigma(L^{\alpha}\left(
			\mathbb{T} \right),L^{\alpha^{'}}\left(
			\mathbb{T} \right))}
		=\overline{\varphi H^{\infty} (\mu)}^{\sigma(L^{\alpha}\left(
			\mathbb{T} \right),L^{\alpha^{'}}\left(
			\mathbb{T} \right))} =\varphi H^{\alpha}(\mathbb{T},\mu)$, for some unimodular function $\varphi$. 
		If $M =\chi_{E} L^{\infty} (\mu)\text{,}$ $W =\overline{W \cap L^{\infty} (\mu)}^{\sigma(L^{\alpha}\left(
			\mathbb{T} \right),L^{\alpha^{'}}\left(
			\mathbb{T} \right))}
		 =\overline{M}^{\sigma(L^{\alpha}\left(
			\mathbb{T} \right),L^{\alpha^{'}}\left(
			\mathbb{T} \right))}
		 =\overline{\chi_{E} L^{\infty} (\mu
			)}^{\sigma(L^{\alpha}\left(
			\mathbb{T} \right),L^{\alpha^{'}}\left(
			\mathbb{T} \right))}
		 =\chi_{E} L^{\alpha} (\mathbb{T},\mu)$, for some Borel
		subset $E$ of $\mathbb{T}$. The proof is completed.
	\end{proof}

	\section{The Extension of Beurling Theorem in Finite von Nuemann Algebras}
	
	Let $\mathcal{M}$ be a finite von Neumann algebra with a faithful, normal,
	tracial state $\tau .$ Given a von Neumann subalgebra $\mathcal{D}$ of $%
	\mathcal{M}$, a conditional expectation $\Phi $: $\mathcal{M}\rightarrow 
	\mathcal{D}$ is a positive linear map satisfying $\Phi (I)=I$ and $\Phi
	(x_{1}yx_{2})=x_{1}\Phi (y)x_{2}$ for all $x_{1},x_{2}\in \mathcal{D}$ and $%
	y\in \mathcal{M}.$ There exists a unique conditional expectation $\Phi
	_{\tau }$: $\mathcal{M}\rightarrow \mathcal{D}$ satisfying $\tau \circ \Phi
	_{\tau }(x)=\tau (x)$ for every $x\in \mathcal{M}$. Now we recall
	noncommutative Hardy spaces $H^{\infty }(\mathcal{M},\tau)$ in \cite{Arveson}.
	
	\begin{definition}
		Let $\mathcal{A}$ be a weak* closed unital subalgebra of $\mathcal{M}$, and
		let $\Phi _{\tau }$ be the unique faithful normal trace preserving
		conditional expectation from $\mathcal{M}$ onto the diagonal von Neumann
		algebra $\mathcal{D}=\mathcal{A}\cap \mathcal{A^{\ast }}$. Then $\mathcal{A}$
		is called a finite, maximal subdiagonal subalgebra of $\mathcal{M}$ with
		respect to $\Phi _{\tau }$ if\newline
		(1) $\mathcal{A}+\mathcal{A^{\ast }}$ is weak* dense in $\mathcal{M}$,%
		\newline
		(2) $\Phi _{\tau }(xy)=\Phi _{\tau }(x)\Phi _{\tau }(y)$ for all $x,y\in 
		\mathcal{A}$.\newline
		Such $\mathcal{A}$ will be denoted by $H^{\infty }(\mathcal{M},\tau)$, and $\mathcal{A}$ is
		also called a noncommutative Hardy space.
	\end{definition}
	
	\begin{example}
		Let $\mathcal{M}=L^{\infty }(\mathbb{T},\mu )$, and $\tau (f)=\int fd\mu $
		for all $f\in L^{\infty }(\mathbb{T},\mu )$. Let $\mathcal{A}=H^{\infty }(%
		\mathbb{T},\mu )$, then $\mathcal{D}=H^{\infty }(\mathbb{T},\mu )\cap
		H^{\infty }(\mathbb{T},\mu )^{\ast }=\mathbb{C}$. Let $\Phi _{\tau }$ be the
		mapping from $L^{\infty }(\mathbb{T},\mu )$ onto $\mathbb{C}$ defined by $%
		\Phi _{\tau }(f)=\int fd\mu $. Then $H^{\infty }(\mathbb{T},\mu )$ is a
		finite, maximal subdiagonal subalgebra of $L^{\infty }(\mathbb{T},\mu )$.
	\end{example}
	
	\begin{example}
		Let $\mathcal{M}=\mathcal{M}_{n}(\mathbb{C})$ be with the usual trace $\tau $%
		. Let $\mathcal{A}$ be the subalgebra of lower triangular matrices, now $%
		\mathcal{D}$ is the diagonal matrices and $\Phi _{\tau }$ is the natural
		projection onto the diagonal matrices. Then $\mathcal{A}$ is a finite maximal
		subdiagonal subalgebra of $\mathcal{M}_{n}(\mathbb{C})$.
	\end{example}
	
	Let $\mathcal{M}$ be a finite von Neumann algebra with a faithful, normal,
	tracial state $\tau $, $\Phi _{\tau }$ be the conditional expectation and $%
	\alpha $ be a normalized, unitarily invariant $\|\|_{1}$-dominating norm
	on $\mathcal{M}$. Let $L^{\alpha }(\mathcal{M},\tau )$ be the $\alpha $
	closure of $\mathcal{M}$,i.e., $L^{\alpha }(\mathcal{M},\tau )=[\mathcal{M}%
	]_{\alpha }$ and $(L^{\alpha }(\mathcal{M},\tau ))^{'}$ be the dual space of $L^{\alpha }(\mathcal{M},\tau )$, more details about the dual space of $L^{\alpha }(\mathcal{M},\tau )$ is in \cite{Chen3}. Similarly, we define $H^{\alpha }(\mathcal{M},\tau )= \overline{H^{\infty }(%
	\mathcal{M},\tau )}^{\sigma(L^{\alpha}\left(
	\mathcal{M},\tau \right),L^{\alpha^{'}}\left(
	\mathcal{M},\tau \right))}\cap L^{\alpha }(\mathcal{M},\tau ) $, $H_{0}^{\infty }(\mathcal{M},\tau )=\ker
	(\Phi _{\tau })\cap H^{\infty }(\mathcal{M},\tau )$ and $H_{0}^{\alpha }(%
	\mathcal{M},\tau )=\ker (\Phi _{\tau })\cap H^{\alpha }(\mathcal{M},\tau ).$
		
	\begin{example}
	Let $\alpha =\Vert \cdot \Vert _{p}$, then $L^{p}(\mathcal{M},\tau )=[%
	\mathcal{M}]_{p}$, $H^{p}(\mathcal{M},\tau )=\overline{H^{\infty }(%
		\mathcal{M},\tau )}^{\sigma(L^{p}\left(
		\mathcal{M},\tau \right),L^{q}\left(
		\mathcal{M},\tau \right))}\cap L^{p }(\mathcal{M},\tau )$. 
	\end{example}
	
	 In \cite{Saito}, K. S. Satio characterized the noncommutative
Hardy spaces $H^{p}(\mathcal{M},\tau )$ and $H_{0}^{p}(\mathcal{M},\tau )$.

\begin{lemma}
		Let $\mathcal{M}$ be a finite
		von Neumann algebra with a faithful, normal, tracial state $\tau $, and then \newline
	(1) $H^{1}(\mathcal{M},\tau )=\{x\in L^{1}(\mathcal{M},\tau ):\tau (xy)=0$
	for all $y\in H_{0}^{\infty }\}$,\newline
	(2) $H_{0}^{1}(\mathcal{M},\tau )=\{x\in L^{1}(\mathcal{M},\tau ):\tau (xy)=0
	$ for all $y\in H^{\infty }\}$,\newline
	(3) $H_{0}^{1}(\mathcal{M},\tau )=\{x\in H^{1}(\mathcal{M},\tau ):\Phi
	_{\tau }(xh)=0\}$.
\end{lemma}

	 \begin{theorem}
	  	\label{Halpha}Let $\mathcal{M}$ be a finite
	  	von Neumann algebra with a faithful, normal, tracial state $\tau $ and $%
	  	\alpha $ be a normalized, unitarily invariant $\Vert \cdot \Vert _{1,\tau }$%
	  	-dominating norm on $\mathcal{M}$. Let $H^{\infty }$ be a finite
	  	subdiagonal subalgebra of $\mathcal{M}$. Then there
	  	exists a faithful normal tracial state $\tau $ such that $H^{\alpha }(%
	  	\mathcal{M},\tau )=H^{1}(\mathcal{M},\tau )\cap L^{\alpha }(\mathcal{M},\tau)$.
	  \end{theorem}
	  \begin{proof}
	  By the definition of $H^{\alpha }(\mathcal{M},\tau )$, we have $H^{\alpha }(\mathcal{M},\tau )= \overline{H^{\infty }(%
	  	\mathcal{M},\tau )}^{\sigma(L^{\alpha}\left(
	  	\mathcal{M},\tau \right),L^{\alpha^{'}}\left(
	  	\mathcal{M},\tau \right))}\cap L^{\alpha }(\mathcal{M},\tau ) \subseteq L^{\alpha}\left(
	  \mathcal{M},\tau \right) $. For every $x\in H^{\alpha }(\mathcal{M},\tau )= \overline{H^{\infty }(%
	  	\mathcal{M},\tau )}^{\sigma(L^{\alpha}\left(
	  	\mathcal{M},\tau \right),L^{\alpha^{'}}\left(
	  	\mathcal{M},\tau \right))}\cap L^{\alpha }(\mathcal{M},\tau ) $, there exists a net ${x_{n}} $ in $H^{\infty }(\mathcal{M},\tau )$ 
	  such that $x_{n}\rightarrow x$ in $\sigma(L^{\alpha}\left(
	  \mathcal{M},\tau \right),L^{\alpha^{'}}\left(
	  \mathcal{M},\tau \right))$ topology. Since $x_{n}\in H^{\infty }(\mathcal{M},\tau )\subseteq H^{1}(\mathcal{M},\tau )$, $\tau(x_{n}y)=0$ for all $y\in H^{\infty }(\mathcal{M},\tau ), \Phi(y)=0$. We know $x_{n}\rightarrow x$ in $\sigma(L^{\alpha}\left(
	  \mathcal{M},\tau \right),L^{\alpha^{'}}\left(
	  \mathcal{M},\tau \right))$ topology, so $\forall y \in H^{\infty }(\mathcal{M},\tau ), \tau(x_{n}y) \rightarrow \tau(xy)$. Therefore, for all $y\in H^{\infty }(\mathcal{M},\tau ), \Phi(y)=0$,$\tau(xy)=0$.
	  Thus  $H^{\alpha }(\mathcal{M},\tau )\subseteq H^{1}(\mathcal{M},\tau )$.Therefore, 
	  $H^{\alpha }(\mathcal{M},\tau )\subseteq H^{1}(\mathcal{M},\tau )\cap L^{\alpha }(\mathcal{M},\tau)$.\\
	  Next, we show that $H^{\alpha }(\mathcal{M},\tau )= H^{1}(\mathcal{M},\tau )\cap L^{\alpha }(\mathcal{M},\tau)$.\\
	   Assume, via contradiction, that $H^{\alpha }(\mathcal{M},\tau )\subsetneq H^{1}(\mathcal{M},\tau )\cap L^{\alpha }(\mathcal{M},\tau)$. By the Hahn-Banach theorem, there is a $\sigma(L^{\alpha}\left(
	  \mathcal{M},\tau \right),L^{\alpha^{'}}\left(
	  \mathcal{M},\tau \right))$-continuous functional $\Phi$ on $L^{\alpha}\left(\mathcal{M},\tau \right)$ and $x\in H^{1}(\mathcal{M},\tau )\cap L^{\alpha }(\mathcal{M},\tau)$ such that 
	  $\Phi(x)=0$ and $\Phi(y)=0$ for $\forall y \in H^{\alpha }(\mathcal{M},\tau )$. Since $\xi \in L^{\alpha^{'}}\left(
	  \mathcal{M},\tau \right) $ such that $\Phi(z)=\tau(z\xi), \forall z \in L^{\alpha }(\mathcal{M},\tau)$, we have $\Phi(y)=\tau(y\xi), \forall y \in H^{\alpha }(\mathcal{M},\tau)\subseteq L^{\alpha }(\mathcal{M},\tau)$. Because $\xi \in L^{\alpha^{'}}\left(
	  \mathcal{M},\tau \right)\subseteq L^{1 }(\mathcal{M},\tau) $ and $\Phi(y)=\tau(y\xi), \forall y \in H^{\infty }(\mathcal{M},\tau)$, $\xi \in H^{1}(\mathcal{M},\tau )_{0}$.\\
	  Since $x\in H^{1}(\mathcal{M},\tau ), \tau(x\xi_{n}) =0, \forall \xi_{n}\in  H^{\infty }(\mathcal{M},\tau)_{0} $. There exists a net $\xi_{n} \in  H^{\infty }(\mathcal{M},\tau)_{0} $ such that $\xi_{n}\rightarrow \xi$ in $\|\|_{1}$ topology, so $\tau(\xi_{n})\rightarrow \tau(\xi)$. By lemma 3.4 in \cite{Chen3}, $\tau(x\xi_{n})\rightarrow \tau(x\xi)$. Therefore, $\Phi(x)=\tau(x\xi)=0$, which contradicts $\Phi(x)=0$. Thus, $H^{\alpha }(\mathcal{M},\tau )= H^{1}(\mathcal{M},\tau )\cap L^{\alpha }(\mathcal{M},\tau)$.
	  \end{proof}

	\begin{theorem}
		\label{factorization} Let $\mathcal{M}$ be a finite von Neumann algebra with a faithful, normal, tracial state $\tau $ and $\alpha $ be a normalized, unitarily invariant, $\|\|_{1}$-dominating norm on $\mathcal{M}$. Let $H^{\infty }$
		be a finite subdiagonal subalgebra of $\mathcal{M}$.
		If $k\in \mathcal{M}$ and $k^{-1}\in L^{\alpha }\left( 
		\mathcal{M},\tau \right) ,$ then there are unitary operators $u_{1},u_{2}\in 
		\mathcal{M}$ and $s_{1},s_{2}\in H^{\infty }$ such that $%
		k=u_{1}s_{1}=s_{2}u_{2}$ and $s_{1}^{-1},s_{2}^{-1}\in H^{\alpha }(\mathcal{M%
		},\tau )$.
	\end{theorem}
	
	\begin{proof}
			Suppose $k\in \mathcal{M}$ with $k^{-1}\in L^{\alpha }(%
			\mathcal{M},\tau).$ Assume that $k=v\left\vert k\right\vert $ is the polar
			decomposition of $k$ in $\mathcal{M}$, where $v$ is a unitary in $\mathcal{M}$. Then from the
		assumption that $k^{-1}=\left\vert k\right\vert^{-1}v^{*}$, so we have $\left\vert k\right\vert^{-1} \in L^{\alpha }(%
		\mathcal{M},\tau)\subset L^{1}(%
		\mathcal{M},\tau) $. Since $\left\vert k\right\vert$ in $\mathcal{M}$ positive, we have $\left\vert k\right\vert^{-\frac{1}{2}}\in L^{2}(\mathcal{M},\tau)$ and $\left\vert k\right\vert^{\frac{1}{2}}\in \mathcal{M} $. There exists a unitary operator$u_{1}\in \mathcal{M}$and $s_{1}\in H^{\infty}$ such that

	\end{proof}
	
	The following density theorem also plays an important role in the proof of our main result of the paper.
	
	\begin{theorem}
		\label{density}Let $\mathcal{M}$ be a finite von Neumann algebra with a faithful, normal, tracial state $\tau $ and $\alpha $ be a 
		normalized, unitarily invariant, $\|\|_{1}$-dominating norm on $\mathcal{M}$. If $\mathcal{W}$ is a $\sigma(L^{\alpha}\left(
		\mathcal{M},\tau \right),L^{\alpha^{'}}\left(
		\mathcal{M},\tau \right))$-closed subspace of $L^{\alpha }(\mathcal{M},\tau )$ and $\mathcal{N}$ is a weak* closed linear subspace of $
		\mathcal{M}$ such that $\mathcal{W}H^{\infty }\subset \mathcal{W}$ and $%
		\mathcal{N}H^{\infty }\subset \mathcal{N}$, then\newline
		(1) $\mathcal{N}=\overline{\mathcal{N}}^{\sigma(L^{\alpha}\left(
			\mathcal{M},\tau \right),L^{\alpha^{'}}\left(
			\mathcal{M},\tau \right))}\cap \mathcal{M}$,\newline
		(2) $\mathcal{W}\cap \mathcal{M}$ is weak* closed in $\mathcal{M}$,\newline
		(3) $\mathcal{W}=\overline{\mathcal{W}\cap \mathcal{M}}^{\sigma(L^{\alpha}\left(
			\mathcal{M},\tau \right),L^{\alpha^{'}}\left(
			\mathcal{M},\tau \right))}$,\newline
		(4) If $\mathcal{S}$ is a subspace of $\mathcal{M}$ such that $\mathcal{S}%
		H^{\infty }\subset \mathcal{S}$, then $\overline{\mathcal{S}}^{\sigma(L^{\alpha}\left(
			\mathcal{M},\tau \right),L^{\alpha^{'}}\left(
			\mathcal{M},\tau \right))}=\overline{\overline{%
			\mathcal{S}}^{w\ast }}^{\sigma(L^{\alpha}\left(
		\mathcal{M},\tau \right),L^{\alpha^{'}}\left(
		\mathcal{M},\tau \right))}$, where $\overline{\mathcal{S}}^{w\ast }$
		is the weak*-closure of $\mathcal{S}$ in $\mathcal{M}$.
	\end{theorem}
	\begin{proof}
		For (1), it is clear that $\mathcal{N}\subseteq \overline{ \mathcal{N}}^{\sigma(L^{\alpha}\left(
		 	\mathcal{M},\tau \right),L^{\alpha^{'}}\left(
		 	\mathcal{M},\tau \right))}\cap \mathcal{M}$. Assume, via contradiction, that $\mathcal{N}\subsetneq
		\overline{\mathcal{N}}^{\sigma(L^{\alpha}\left(
			\mathcal{M},\tau \right),L^{\alpha^{'}}\left(
			\mathcal{M},\tau \right))}\cap \mathcal{M}$. Note that $\mathcal{N}$ is
		a weak* closed linear subspace of $\mathcal{M}$ and $L^{1}(\mathcal{M},\tau )
		$ is the predual space of $(\mathcal{M},\tau )$. It follows from the
		Hahn-Banach theorem that there exist a $\xi \in L^{1}(\mathcal{M},\tau )$
		and an $x\in \overline{\mathcal{N}}^{\sigma(L^{\alpha}\left(
			\mathcal{M},\tau \right),L^{\alpha^{'}}\left(
			\mathcal{M},\tau \right))}\cap \mathcal{M}$ such that\newline
		(a) $\tau (\xi x)\neq 0$ and
		(b) $\tau (\xi y)=0$ for all $y\in \mathcal{N}$.\newline
		We claim that there exists a $z\in \mathcal{M}$ such that\newline
		(a$^{\prime }$) $\tau (zx)\neq 0$ and
		(b$^{\prime }$) $\tau (zy)=0$ for all $y\in \mathcal{N}$. Actually assume
		that $\xi =|\xi ^{\ast }|v$ is the polar decomposition of $\xi \in L^{1}(%
		\mathcal{M},\tau )$, where $v$ is a unitary element in $\mathcal{M}$ and $%
		|\xi ^{\ast }|$ is in $L^{1}(\mathcal{M},\tau )$ is positive. Let $f$ be a
		function on $[0,\infty )$ defined by the formula $f(t)=1$ for $0\leq t\leq 1$
		and $f(t)=1/t$ for $t>1$. We define $k=f(|\xi ^{\ast }|)$ by the functional
		calculus. Then by the construction of $f$, we know that $k\in \mathcal{M}$
		and $k^{-1}=f^{-1}(|\xi ^{\ast }|)\in L^{1}(\mathcal{M},\tau )$. It follows
		from theorem \ref{factorization} that there exist a unitary operator $u\in 
		\mathcal{M}$ and $s\in H^{\infty }$ such that $k=us$ and $s^{-1}\in H^{1}(%
		\mathcal{M},\tau )$. Therefore, we can further assume that $\{{t_{n}\}}%
		_{n=1}^{\infty }$ is a sequence of elements in $H^{\infty }$ such that $%
		\Vert s^{-1}-t_{n}\Vert _{1,\tau }\rightarrow 0$. Observe that\newline
		(i) Since $s,t_{n}$ are in $H^{\infty }$, for each $y\in \mathcal{N}$ we
		have that $yt_{n}s\in \mathcal{N}H^{\infty }\subseteq \mathcal{N}$ and $\tau
		(t_{n}s\xi y)=\tau (\xi yt_{n}s)=0$,\newline
		(ii) We have $s\xi =(u\ast u)s(|\xi ^{\ast }|v)=u\ast (k|\xi ^{\ast }|)v\in 
		\mathcal{M}$, by the definition of $k$,\newline
		(iii) From (a) and (i), we have $0\neq \tau (\xi x)=\tau (s^{-1}s\xi x)=%
		\underset{n\rightarrow \infty }{\lim }\tau (t_{n}s\xi x)$.\newline
		Combining (i), (ii) and (iii), we are able to find an $N\in \mathbb{Z}$ such
		that $z=t_{N}s\xi \in \mathcal{M}$ satisfying\newline
		(a$^{\prime }$) $\tau (zx)\neq 0$ and
		(b$^{\prime }$) $\tau (zy)=0$ for all $y\in \mathcal{N}$.\newline
		Recall that $x\in \overline{\mathcal{N}}^{\sigma(L^{\alpha}\left(
			\mathcal{M},\tau \right),L^{\alpha^{'}}\left(
			\mathcal{M},\tau \right))}$. Then there is
		a sequence $\{{x_{n}\}}\subseteq \mathcal{N}$ such that $x_{n}\rightarrow x$ in $\sigma(L^{\alpha}\left(
		\mathcal{M},\tau \right),L^{\alpha^{'}}\left(
		\mathcal{M},\tau \right))$ topology. We have\newline
		$|\tau (zx_{n})-\tau (zx)|\rightarrow 0$.\newline
		Combining with (b$^{\prime }$) we conclude that $\tau (zx)=\underset{%
			n\rightarrow \infty }{\lim }\tau (zx_{n})=0$. This contradicts with the
		result (a$^{\prime }$). Therefore, $\mathcal{N}=\overline{\mathcal{N}}^{\sigma(L^{\alpha}\left(
			\mathcal{M},\tau \right),L^{\alpha^{'}}\left(
			\mathcal{M},\tau \right))}\cap 
		\mathcal{M}$.
		
		For (2), let $\overline{\mathcal{W}\cap \mathcal{M}}^{w\ast }$ be the
		weak*-closure of $\mathcal{W}\cap \mathcal{M}$ in $\mathcal{M}$. In order to
		show that $\mathcal{W}\cap \mathcal{M}=\overline{\mathcal{W}\cap \mathcal{M}}%
		^{w\ast }$, it suffices to show that $\overline{\mathcal{W}\cap \mathcal{M}}%
		^{w\ast }\subseteq \mathcal{W}$. Assume, to the contrary, that $\overline{%
			\mathcal{W}\cap \mathcal{M}}^{w\ast }\nsubseteq \mathcal{W}$. Thus there
		exists an element $x$ in $\overline{\mathcal{W}\cap \mathcal{M}}^{w\ast
		}\subset \mathcal{M}\subseteq L^{\alpha }(\mathcal{M},\tau )$, but $x\notin 
		\mathcal{W}$. Since $\mathcal{W}$ is a $\sigma(L^{\alpha}\left(
		\mathcal{M},\tau \right),L^{\alpha^{'}}\left(
		\mathcal{M},\tau \right))$-closed subspace of $L^{\alpha }(%
		\mathcal{M},\tau )$, by the Hahn-Banach theorem, there exists a $\xi \in
		L^{1}(\mathcal{M},\tau )$ such that $\tau (\xi x)\neq 0$ and $\tau (\xi y)=0$
		for all $y\in \mathcal{W}$. Since $\xi \in L^{1}(\mathcal{M},\tau )$, the
		linear mapping $\tau _{\xi }:\mathcal{M}\rightarrow \mathcal{%
			\mathbb{C}
		}$, defined by $\tau _{\xi }(a)=\tau (\xi a)$ for all $a\in \mathcal{M}$ is
		weak*-continuous. Note that $x\in \overline{\mathcal{W}\cap \mathcal{M}}%
		^{w\ast }$ and $\tau (\xi y)=0$ for all $y\in \mathcal{W}$. We know that $%
		\tau (\xi x)=0$, which contradicts with the assumption that $\tau (\xi
		x)\neq 0$. Hence $\overline{\mathcal{W}\cap \mathcal{M}}^{w\ast }\subseteq 
		\mathcal{W}$, so $\mathcal{W}\cap \mathcal{M}=\overline{\mathcal{W}\cap 
			\mathcal{M}}^{w\ast }$.
		
		For (3), since $\mathcal{W}$ is $\sigma(L^{\alpha}\left(
		\mathcal{M},\tau \right),L^{\alpha^{'}}\left(
		\mathcal{M},\tau \right))$-closed, we have $\overline{
		\mathcal{W}\cap \mathcal{M}}^{\sigma(L^{\alpha}\left(
		\mathcal{M},\tau \right),L^{\alpha^{'}}\left(
		\mathcal{M},\tau \right))}\subseteq \mathcal{W}$. Now we assume $%
		\overline{\mathcal{W}\cap \mathcal{M}}^{\sigma(L^{\alpha}\left(
			\mathcal{M},\tau \right),L^{\alpha^{'}}\left(
			\mathcal{M},\tau \right))}\subsetneq \mathcal{W}\subseteq
		L^{\alpha }(\mathcal{M},\tau )$. By the Hahn-Banach theorem, there exists an $%
		x\in \mathcal{W}$ and $\xi \in L^{1}(\mathcal{M},\tau )$ such that $\tau
		(\xi x)\neq 0$ and $\tau (\xi y)=0$ for all $y\in \overline{\mathcal{W}\cap 
		\mathcal{M}}^{\sigma(L^{\alpha}\left(
		\mathcal{M},\tau \right),L^{\alpha^{'}}\left(
		\mathcal{M},\tau \right))}$. Let $x=v|x|$ be the polar decomposition of $x$ in $%
		L^{\alpha }(\mathcal{M},\tau )$, where $v$ is a unitary element in $\mathcal{%
			M}$. Let $f$ be a function on $[0,\infty )$ defined by the formula $f(t)=1$
		for $0\leq t\leq 1$ and $f(t)=1/t$ for $t>1$. We define $k=f(|x|)$ by the
		functional calculus. Then by the construction of $f$, we know that $k\in 
		\mathcal{M}$ and $k^{-1}=f^{-1}(|x|)\in L^{\alpha }(\mathcal{M},\tau )$. It
		follows from theorem \ref{factorization} that there exist a unitary operator 
		$u\in \mathcal{M}$ and $s\in H^{\infty }$ such that $k=su$ and $s^{-1}\in
		H^{\alpha }(\mathcal{M},\tau )$. A little computation shows that $|x|k\in 
		\mathcal{M}$ which implies that $xs=xsuu^{\ast }=xku^{\ast }=v(|x|k)u^{\ast
		}\in \mathcal{M}$. Since $s\in H^{\infty }$, we know $xs\in \mathcal{W}%
		H^{\infty }\subseteq \mathcal{W}$ and thus $xs\in \mathcal{W}\cap \mathcal{M}
		$. Furthermore, note that $(\mathcal{W}\cap \mathcal{M})H^{\infty }\subseteq 
		\mathcal{W}\cap \mathcal{M}$. Thus, if $t\in H^{\infty }$ we see $xst\in 
		\mathcal{W}\cap \mathcal{M}$, and $\tau (\xi xst)=0$.
		 Since $H^{\alpha }(\mathcal{M},\tau )= \overline {H^{\infty}}^{\sigma(L^{\alpha}\left(
		 	\mathcal{M},\tau \right),L^{\alpha^{'}}\left(
		 	\mathcal{M},\tau \right))}\cap L^{\alpha }(\mathcal{M},\tau )$,
		$\forall \in H^{\alpha }(\mathcal{M},\tau ) $ and there is a net ${t_{n}}$ in $ H^{\infty}$ such that $t_{n}\rightarrow t$ in $\sigma(L^{\alpha}\left(
		\mathcal{M},\tau \right),L^{\alpha^{'}}\left(
		\mathcal{M},\tau \right))$ topology. We have $\xi xs \in L^{\alpha^{'} }(\mathcal{M},\tau ) $ because $\alpha^{'}(\xi xs)\leq \alpha^{'}(\xi)\|xs\|$. Therefore, $\tau (\xi xst_{n})\rightarrow \tau (\xi xst)$, which follows that $\tau (\xi xst)=0$ for all $t\in H^{\alpha }(%
		\mathcal{M},\tau )$. Since $s^{-1}\in H^{\alpha }(\mathcal{M},\tau )$, we
		see that $\tau (\xi x)=\tau (\xi xss^{-1})=0$. This contradicts with the
		assumption that $\tau (\xi x)\neq 0$ . Therefore $\mathcal{W}=\overline{\mathcal{W}%
		\cap \mathcal{M}}^{\sigma(L^{\alpha}\left(
		\mathcal{M},\tau \right),L^{\alpha^{'}}\left(
		\mathcal{M},\tau \right))}$.
		
		For (4), assume that $\mathcal{S}$ is a subspace of $\mathcal{M}$ such that $%
		\mathcal{S}H^{\infty }\subset \mathcal{S}$ and $\overline{\mathcal{S}}%
		^{w\ast }$ is weak*-closure of $\mathcal{S}$ in $\mathcal{M}$. Then $\overline{\mathcal{S}}^{\sigma(L^{\alpha}\left(
			\mathcal{M},\tau \right),L^{\alpha^{'}}\left(
			\mathcal{M},\tau \right))}H^{\infty }\subseteq 
		\overline{ \mathcal{S}}^{\sigma(L^{\alpha}\left(
			\mathcal{M},\tau \right),L^{\alpha^{'}}\left(
			\mathcal{M},\tau \right))}$.
		Note that $\mathcal{S}\subseteq \overline{\mathcal{S}}^{\sigma(L^{\alpha}\left(
			\mathcal{M},\tau \right),L^{\alpha^{'}}\left(
			\mathcal{M},\tau \right))}\cap \mathcal{M%
		}$. From (2), we know that $\overline{\mathcal{S}}^{\sigma(L^{\alpha}\left(
		\mathcal{M},\tau \right),L^{\alpha^{'}}\left(
		\mathcal{M},\tau \right))}\cap \mathcal{M}$ is
		weak*-closed. Therefore, $\overline{\mathcal{S}}^{w\ast }\subseteq \overline{\mathcal{S}}^{\sigma(L^{\alpha}\left(
			\mathcal{M},\tau \right),L^{\alpha^{'}}\left(
			\mathcal{M},\tau \right))}\cap \mathcal{M}$.\\
		 Since $\overline{\overline{\mathcal{S}}%
		^{w\ast }}^{\sigma(L^{\alpha}\left(
		\mathcal{M},\tau \right),L^{\alpha^{'}}\left(
		\mathcal{M},\tau \right))}\subseteq\overline{ \overline{\mathcal{S}}^{\sigma(L^{\alpha}\left(
		\mathcal{M},\tau \right),L^{\alpha^{'}}\left(
		\mathcal{M},\tau \right))}\cap \mathcal{M}}^{\sigma(L^{\alpha}\left(
	\mathcal{M},\tau \right),L^{\alpha^{'}}\left(
	\mathcal{M},\tau \right))}= \overline{\mathcal{S}}^{\sigma(L^{\alpha}\left(
		\mathcal{M},\tau \right),L^{\alpha^{'}}\left(
		\mathcal{M},\tau \right))}$, we have $\overline{\mathcal{%
			S}}^{\sigma(L^{\alpha}\left(
		\mathcal{M},\tau \right),L^{\alpha^{'}}\left(
		\mathcal{M},\tau \right))}=\overline{\overline{\mathcal{S}}^{w\ast }}^{\sigma(L^{\alpha}\left(
		\mathcal{M},\tau \right),L^{\alpha^{'}}\left(
		\mathcal{M},\tau \right))}$.
	\end{proof}
	
	 Before we obtain our main result in the paper, we call the definitions of internal column sum of a family of subspaces, and
	 the lemma in \cite{B.L.}.
	 
	 \begin{definition}
	 	\text{(from }\cite{B.L.}\text{)} Let $\mathcal{M}$ be a finite
	 	von Neumann algebra with a faithful, normal, tracial state $\tau $ and $\alpha $ be a 
	 	normalized, unitarily invariant, $\|\centerdot\|_{1}$-dominating norm. Suppose $X$ be a closed subspace of $L^{\alpha }(\mathcal{M},\tau )$ with $%
	 	\alpha \in N_{\Delta }\left( \mathcal{M},\tau \right) $. Then $X$ is called
	 	an internal column sum of a family of closed subspaces $\{X_{\lambda
	 	}\}_{\lambda \in \Lambda }$ of $L^{\alpha }(\mathcal{M},\tau )$, denoted by 
	 	$X=\bigoplus_{\lambda \in \Lambda }^{col}X_{\lambda }$ if\newline
	 	(1) $X_{\mu }^{\ast }X_{\lambda }=\{0\}$ for all distinct $\lambda ,\mu \in
	 	\Lambda $, and\newline
	 	(2) $X=\overline{span\{X_{\lambda }:\lambda \in \Lambda \}}^{\sigma(L^{\alpha}\left(
	 		\mathcal{M},\tau \right),L^{\alpha^{'}}\left(
	 		\mathcal{M},\tau \right))}$.
	 \end{definition}
	
	\begin{lemma}
		\label{colsum}\text{(from }\cite{B.L.}\text{)} Let $\mathcal{M}$ be a finite
		von Neumann algebra with a faithful, normal, tracial state $\tau $ and $%
		\alpha $ be a normalized, unitarily invariant $\Vert \cdot \Vert _{1 }$%
		-dominating norm on $\mathcal{M}$. Let $H^{\infty }$ be a finite
		subdiagonal subalgebra of $\mathcal{M}$ and $\mathcal{D}=H^{\infty }\cap
		(H^{\infty })^{\ast }$. Assume that $\mathcal{W}\subseteq \mathcal{M}$ is a
		weak*-closed subspace such that $\mathcal{W}H^{\infty }\subseteq \mathcal{W}$
		. Then there exists a weak*-closed subspace $\mathcal{Y}$ of $\mathcal{M}$
		and a family $\{u_{\lambda }\}_{\lambda \in \Lambda }$of partial isometries
		in $\mathcal{M}$ such that\newline
		(1) $u_{\lambda }^{\ast }\mathcal{Y}=0$ for all $\lambda \in \Lambda $,%
		\newline
		(2) $u_{\lambda }^{\ast }u_{\lambda }\in \mathcal{D}$ and $u_{\lambda
		}^{\ast }u_{\mu }=0$ for all $\lambda ,\mu \in \Lambda $ with $\lambda \neq
		\mu $,\newline
		(3) $\mathcal{Y}=\overline{H_{0}^{\infty }\mathcal{Y}}^{w\ast }$,\newline
		(4) $\mathcal{W}=\mathcal{Y}\oplus ^{col}(\oplus _{\lambda \in
			\Lambda }^{col}u_{\lambda }H^{\infty })$.
	\end{lemma}
	
	Now we are ready to prove our main result of the paper, the
	generalized Beurling Theorem  for noncommutative Hardy spaces associated with finite von Neumann algebras.
	
	\begin{theorem}
		\label{main}Let $\mathcal{M}$ be a finite von Neumann algebra with a
		faithful, normal, tracial state $\tau $ and $\alpha $ be a 
		normalized, unitarily invariant, $\|\|_{1}$-dominating norm on $\mathcal{M}$.Let $H^{\infty }$
		be a finite subdiagonal subalgebra of $\mathcal{M}$ and $\mathcal{D}%
		=H^{\infty }\cap (H^{\infty })^{\ast }$. If $\mathcal{W}$ is a closed
		subspace of $L^{\alpha }(\mathcal{M},\tau )$ such that $\mathcal{W}H^{\infty
		}\subseteq \mathcal{W}$, then there exists a $\sigma(L^{\alpha}\left(
		\mathcal{M},\tau \right),L^{\alpha^{'}}\left(
		\mathcal{M},\tau \right))$ closed subspace $\mathcal{Y}$
		of $L^{\alpha }(\mathcal{M},\tau )$ and a family $\{u_{\lambda }\}_{\lambda
			\in \Lambda }$ of partial isometries in $\mathcal{M}$ such that\newline
		(1) $u_{\lambda }^{\ast }\mathcal{Y}=0$ for all $\lambda \in \Lambda $,%
		\newline
		(2) $u_{\lambda }^{\ast }u_{\lambda }\in \mathcal{D}$ and $u_{\lambda
		}^{\ast }u_{\mu }=0$ for all $\lambda ,\mu \in \Lambda $ with $\lambda \neq
		\mu $,\newline
		(3) $\mathcal{Y}=\overline{H_{0}^{\infty }\mathcal{Y}}^{\sigma(L^{\alpha}\left(
			\mathcal{M},\tau \right),L^{\alpha^{'}}\left(
			\mathcal{M},\tau \right))}$,\newline
		(4) $\mathcal{W}=\mathcal{Y}\oplus ^{col}(\oplus _{\lambda \in
			\Lambda }^{col}u_{\lambda }H^{\alpha }).$
	\end{theorem}
	
	\begin{proof}
		Suppose $\mathcal{W}$ is a closed subspace of $
		L^{\alpha }(\mathcal{M},\tau )$ such that $\mathcal{W}H^{\infty }\subset 
		\mathcal{W}$. Then it follows from part(2) of the theorem \ref{density} that 
		$\mathcal{W}\cap \mathcal{M}$ is weak* closed in $(\mathcal{M},\tau ) $, we also notice $L^{\infty }(\mathcal{M},\tau )=\mathcal{%
			M}$, and $H^{\alpha }(\mathcal{M},\tau )= \overline {H^{\infty}}^{\sigma(L^{\alpha}\left(
			\mathcal{M},\tau \right),L^{\alpha^{'}}\left(
			\mathcal{M},\tau \right))}\cap L^{\alpha }(\mathcal{M},\tau )$. It follows from the lemma \ref{colsum} that%
		\begin{equation*}
		\mathcal{W}\cap \mathcal{M}=\mathcal{Y}_{1}\bigoplus^{col%
		}(\bigoplus_{i\in \mathcal{I}}^{col}u_{i}H^{\infty }),
		\end{equation*}%
		where $\mathcal{Y}_{1}$ is a closed subspace of $L^{\infty }(\mathcal{M}%
		,\tau )$ such that $\mathcal{Y}_{1}=\overline{\mathcal{Y}_{1}H_{0}^{\infty }}%
		^{w\ast }$, and where $u_{i}$ are partial isometries in $\mathcal{W}\cap 
		\mathcal{M}$ with $u_{j}^{\ast }u_{i}=0$ if $i\neq j$ and with $u_{i}^{\ast
		}u_{i}\in \mathcal{D}$. Moreover, for each $i,u_{i}^{\ast }\mathcal{Y}%
		_{1}=\{0\}$, left multiplication by the $u_{i}u_{i}^{\ast }$ are contractive
		projections from $\mathcal{W}\cap \mathcal{M}$ onto the summands $%
		u_{i}H^{\infty }$, and left multiplication by $I-\sum_{i}u_{i}u_{i}^{\ast }$
		is a contractive projection from $\mathcal{W}\cap \mathcal{M}$ onto $%
		\mathcal{Y}_{1}$.
		
		Let $\mathcal{Y}=\overline{\mathcal{Y}_{1}}^{\sigma(L^{\alpha}\left(
			\mathcal{M},\tau \right),L^{\alpha^{'}}\left(
			\mathcal{M},\tau \right))}$. It is not hard to verify that
		for each $i,u_{i}^{\ast }\mathcal{M}=\{0\}$. We also claim that $%
		\overline {u_{i}H^{\infty }}^{\sigma(L^{\alpha}\left(
			\mathcal{M},\tau \right),L^{\alpha^{'}}\left(
			\mathcal{M},\tau \right))}\cap L^{\alpha }(\mathcal{M},\tau )=u_{i}H^{\alpha }=u_{i}(\overline {H^{\infty}}^{\sigma(L^{\alpha}\left(
			\mathcal{M},\tau \right),L^{\alpha^{'}}\left(
			\mathcal{M},\tau \right))}\cap L^{\alpha }(\mathcal{M},\tau ))$. In fact it is obvious that $%
		\overline {u_{i}H^{\infty }}^{\sigma(L^{\alpha}\left(
			\mathcal{M},\tau \right),L^{\alpha^{'}}\left(
			\mathcal{M},\tau \right))}\cap L^{\alpha }(\mathcal{M},\tau )\supseteq u_{i}H^{\alpha }$. We will need only
		to show that $\overline {u_{i}H^{\infty }}^{\sigma(L^{\alpha}\left(
			\mathcal{M},\tau \right),L^{\alpha^{'}}\left(
			\mathcal{M},\tau \right))}\cap L^{\alpha }(\mathcal{M},\tau ) \subseteq u_{i}H^{\alpha }$. Suppose $x\in \overline{
		u_{i}H^{\infty }}^{\sigma(L^{\alpha}\left(
		\mathcal{M},\tau \right),L^{\alpha^{'}}\left(
		\mathcal{M},\tau \right)) }\cap L^{\alpha }(\mathcal{M},\tau )$, there is a net $%
		\{{x_{n}\}}_{n=1}^{\infty }\subseteq H^{\infty }$ such that $u_{i}x_{n} \rightarrow x$ in $ \sigma(L^{\alpha}\left(
		\mathcal{M},\tau \right),L^{\alpha^{'}}\left(
		\mathcal{M},\tau \right))$ topology. By the choice of $u_{i},$ we know that $u_{i}^{\ast}u_{i}\in \mathcal{D}\subseteq
		H^{\infty }$, so $u_{i}^{\ast}u_{i}x_{n}\in H^{\infty }$ for each $n\geq 1$. 
		So $u_{i}^{\ast }u_{i}x_{n} \rightarrow u_{i}^{\ast }x$ in $ \sigma(L^{\alpha}\left(
		\mathcal{M},\tau \right),L^{\alpha^{'}}\left(
		\mathcal{M},\tau \right))$ topology, we obtain that $u_{i}^{\ast }x\in  \overline {H^{\infty}}^{\sigma(L^{\alpha}\left(
			\mathcal{M},\tau \right),L^{\alpha^{'}}\left(
			\mathcal{M},\tau \right))}\cap L^{\alpha }(\mathcal{M},\tau )=
		H^{\alpha }\left(
		\mathcal{M},\tau \right)$. Again from the choice of $u_{i}$, we know that $%
		u_{i}u_{i}^{\ast }u_{i}x_{n}=u_{i}x_{n}$ for each $n\geq 1$. This implies
		that $x=u_{i}(u_{i}^{\ast }x)\in u_{i}H^{\alpha }$. Thus we conclude that $\overline {u_{i}H^{\infty }}^{\sigma(L^{\alpha}\left(
			\mathcal{M},\tau \right),L^{\alpha^{'}}\left(
			\mathcal{M},\tau \right))}\subseteq u_{i}H^{\alpha }$.
		So $\overline {u_{i}H^{\infty }}^{\sigma(L^{\alpha}\left(
			\mathcal{M},\tau \right),L^{\alpha^{'}}\left(
			\mathcal{M},\tau \right))}= u_{i}H^{\alpha }$. Now from parts (3) and (4) of the theorem %
		\ref{density} and from the definition of internal column sum, it follows that%
		\begin{align*}
		\mathcal{W}& =\overline {\mathcal{W}\cap \mathcal{M}}^{\sigma(L^{\alpha}\left(
			\mathcal{M},\tau \right),L^{\alpha^{'}}\left(
			\mathcal{M},\tau \right))}=\overline{\overline{span\{%
			\mathcal{Y}_{1},u_{i}H^{\infty }:i\in \mathcal{I}\}}^{\ast }}^{\sigma(L^{\alpha}\left(
		\mathcal{M},\tau \right),L^{\alpha^{'}}\left(
		\mathcal{M},\tau \right))}\\
	   & =\overline {span\{\mathcal{Y}_{1},u_{i}H^{\infty }:i\in \mathcal{I}\}}^{\sigma(L^{\alpha}\left(
		\mathcal{M},\tau \right),L^{\alpha^{'}}\left(
		\mathcal{M},\tau \right))} \\
		& =\overline{span\{\mathcal{Y},u_{i}H^{\alpha }:i\in \mathcal{I}\}}^{\sigma(L^{\alpha}\left(
			\mathcal{M},\tau \right),L^{\alpha^{'}}\left(
			\mathcal{M},\tau \right))}\\
			&=\mathcal{Y}\bigoplus^{col}(\bigoplus_{i\in \mathcal{I}}^{col%
		}u_{i}H^{\alpha }).
		\end{align*}%
		\newline
		
		Next, we will verify that $\mathcal{Y}=\overline{\mathcal{Y}H_{0}^{\infty }}^{\sigma(L^{\alpha}\left(
				\mathcal{M},\tau \right),L^{\alpha^{'}}\left(
				\mathcal{M},\tau \right))} $. Recall that $\mathcal{Y}=\overline{\mathcal{Y}_{1}}^{\sigma(L^{\alpha}\left(
		\mathcal{M},\tau \right),L^{\alpha^{'}}\left(
		\mathcal{M},\tau \right))}$. It follows from
		part (1) of the theorem \ref{density}, we have\newline
		\begin{equation*}
		\overline {\mathcal{Y}_{1}H_{0}^{\infty }}^{\sigma(L^{\alpha}\left(
			\mathcal{M},\tau \right),L^{\alpha^{'}}\left(
			\mathcal{M},\tau \right))}\cap \mathcal{M}=\overline{%
			\mathcal{Y}_{1}H_{0}^{\infty }}^{w\ast }=\mathcal{Y}_{1}.
		\end{equation*}%
		Hence from part (3) of the theorem \ref{density} we have that\newline
			\begin{align*}
		\mathcal{Y}&\supseteq \overline {\mathcal{Y}H_{0}^{\infty }}^{\sigma(L^{\alpha}\left(
			\mathcal{M},\tau \right),L^{\alpha^{'}}\left(
			\mathcal{M},\tau \right))}\supseteq
		\overline{\mathcal{Y}_{1}H_{0}^{\infty }}^{\sigma(L^{\alpha}\left(
			\mathcal{M},\tau \right),L^{\alpha^{'}}\left(
			\mathcal{M},\tau \right))} \\
		&= \overline{\overline{\mathcal{Y}%
		_{1}H_{0}^{\infty }}^{\sigma(L^{\alpha}\left(
		\mathcal{M},\tau \right),L^{\alpha^{'}}\left(
		\mathcal{M},\tau \right))}\cap \mathcal{M}}^{\sigma(L^{\alpha}\left(
	\mathcal{M},\tau \right),L^{\alpha^{'}}\left(
	\mathcal{M},\tau \right))}\\
        &=\overline{\mathcal{Y}%
		_{1}}^{\sigma(L^{\alpha}\left(
		\mathcal{M},\tau \right),L^{\alpha^{'}}\left(
		\mathcal{M},\tau \right))}\\
	    &=\mathcal{Y}.
		\end{align*}%
		Thus $\mathcal{Y}=\overline{\mathcal{Y}H_{0}^{\infty }}^{\sigma(L^{\alpha}\left(
			\mathcal{M},\tau \right),L^{\alpha^{'}}\left(
			\mathcal{M},\tau \right))} $. Moreover, it is
		not difficult to verify that for each $i$, left multiplication by the $%
		u_{i}u_{i}^{\ast }$ are contractive projections from $\mathcal{K}$ onto the
		summands $u_{i}H^{\alpha }$, and left multiplication by $I-%
		\sum_{i}u_{i}u_{i}^{\ast }$ is a contractive projection from $\mathcal{W}$
		onto $\mathcal{Y}$. Now the proof is completed.
	\end{proof}

\end{document}